\documentclass[11pt]{amsart}
\usepackage{geometry}                
\usepackage{graphicx}
\usepackage{amssymb}
\usepackage{epstopdf}

\title{AK-ASD Twistor}

\begin{document}

\title{Almost-K\"ahler four manifolds with nonnegative biorthogonal curvature}


\author{Inyoung Kim}
\maketitle
  \begin{abstract}
We classify compact almost-K\"ahler four manifolds with nonnegative biorthogonal curvature. 
\end{abstract}

\maketitle

\section{\large\textbf{Introduction}}\label{S:Intro}
Let $(M, \omega)$ be a symplectic four manifold. 
Then there always exists an almost-complex structure $J$ such that 
$\omega(Jv, Jw)=\omega(v, w)$ and $\omega(v, Jv)>0$ for all nonzero $v\in TX$ [19]. 
We define the riemannian metric $g$ by 
\[g(v, w)=\omega(v, Jw).\]
Then $(M, g, \omega, J)$ is called an almost-K\"ahler structure.

Let $\{e_{1}, e_{2}, e_{3}, e_{4}\}$ be an orthonormal basis of $T_{p}M$ for $p\in M$
and let $P$ be the plane generated by $\{e_{1}, e_{2}\}$ and $P^{\perp}$ be the orthogonal plane to $P$
such that $P^{\perp}$ is generated by $\{e_{3}, e_{4}\}$.

Then the biorthogonal curvature is defined by 
\[K^{\perp}(P)=\frac{K(P)+K(P^{\perp})}{2},\]
where $K(P)$ is the sectional curvature on the plane $P$.
In this article, we consider compact almost-K\"ahler four manifolds with $K^{\perp}\geq 0$.

On an oriented riemannian four manifold, the curvature operator according to the decomposition of $\Lambda^{2}=\Lambda^{+}\oplus \Lambda^{-}$
is given by

\[ 
\mathfrak{R}=
\LARGE
\begin{pmatrix}

 \begin{array}{c|c}
\scriptscriptstyle{W_{+}\hspace{5pt}\scriptstyle{+}\hspace{5pt}\frac{s}{12}}I& \scriptscriptstyle{ric_{0}}\\
 \hline
 
 \hspace{10pt}
 
 \scriptscriptstyle{ric_{0}^{*}}& \scriptscriptstyle{W_{-}\hspace{5pt}\scriptstyle{+}\hspace{5pt}\frac{s}{12}}I\\
 \end{array}
 \end{pmatrix}
 \]
Here $ric_{0}$ is the trace-free Ricci part of the curvature operator. 
When $W_{-}=0$, $(M, g)$ is called to be self-dual.

\vspace{20pt}
\hrule
\vspace{10pt}
$\mathbf{Keywords}$: four manifold, harmonic 2-form, biorthogonal curvature, almost-K\"ahler

$\mathbf{MSC}$: 53C21, 53C24, 53C55

Republic of Korea

Email address: kiysd5@gmail.com

$\Lambda^{+}$ is the $(+1)$-eigenspace of Hodge star operator $*$
and $\Lambda^{-}$ is the $(-1)$-eigenspace of $*$.
We define the Hodge-star operator $*$ by 
$\left<*\omega_{1}, \omega_{2}\right>d\mu=\omega_{1}\wedge\omega_{2}$.

\vspace{20pt}

K\"ahler surfaces with positive sectional curvature, more generally positive bisectional curvature are biholomorphic to $\mathbb{CP}_{2}$ [6], [7]. 
The result in [10] implies the classification of compact K\"ahler surfaces with positive orthogonal holomorphic bisectional curvature. 
Almost-K\"ahler four manifolds with positive biorthogonal curvature is diffeomorphic to $\mathbb{CP}_{2}$ [11]. 
Compact K\"ahler surfaces with nonnegative holomorphic bisectional curvature were classified [9].
In this article, we show a similar result can be obtained for compact almost-K\"ahler four manifolds with nonnegative biorthogonal curvature.
Moreover, considering Gray's result [8], we show that a compact almost-K\"ahler four manifold $(M, g, \omega)$ with $J$-invariant Ricci tensor 
is locally symmetric if $(M, g)$ has nonnegative biorthogonal curvature and constant scalar curvature.

\vspace{50pt}

\section{\large\textbf{Self-dual almost-K\"ahler four manifolds with scalar curvature conditions}}\label{S:Intro}
We consider compact self-dual almost-K\"ahler four manifolds with scalar curvature conditions. 
 
\newtheorem{Proposition}{Proposition}
\begin{Proposition}
Let $(M, g, \omega)$ be a compact almost-K\"ahler four-manifold with zero scalar curvature. 
Then either $(M, g, \omega)$ is scalar-flat K\"ahler or $M$ is diffeomorphic to a rational or ruled surface. 
In particular, if $c_{1}\cdot\omega=0$, then $(M, g, \omega)$ is scalar-flat K\"ahler. 

\end{Proposition}
\begin{proof}
For a four-dimensional almost-K\"ahler metric, we have 
\[s^{*}:=2R(\omega, \omega)=2\left(W_{+}(\omega, \omega)+\frac{s}{12}|\omega|^{2}\right).\]
From a Weitzenb\"ock formula for self-dual 2-forms, $\Delta\alpha=\nabla^{*}\nabla\alpha-2W_{+}(\alpha, \cdot)+\frac{s}{3}\alpha$, 
we have
\[0=\frac{1}{2}\Delta|\omega|^{2}+|\nabla\omega|^{2}-2W_{+}(\omega, \omega)+\frac{s}{3}|\omega|^{2}\]
\[0=|\nabla\omega|^{2}-2W_{+}(\omega, \omega)+\frac{s}{3}|\omega|^{2},\]
since $\omega$ is a self-dual harmonic 2-form of length $\sqrt{2}$. 
From this, we get 
\[s^{*}=s+|\nabla\omega|^{2}.\]
Thus $s^{*}\geq 0$ if $s\equiv0$. 
Moreover, $s^{*}\equiv 0$ if and only if $(M, g, \omega)$ is K\"ahler. 
Note that if $s^{*}\not\equiv 0$, then $\int_{M}\left(\frac{s+s^{*}}{2}\right)d\mu>0$.
From the following formula, 
\[4\pi c_{1}\cdot[\omega]=\int_{M}\frac{s+s^{*}}{2}d\mu,\]
we get $c_{1}\cdot[\omega]>0$. 
They by [18], [24], $(M, \omega)$ is a rational or ruled surface. 

If $c_{1}\cdot\omega=0$, then 
\[\int_{M}\left(\frac{s+s^{*}}{2}\right)d\mu=\int_{M}\frac{s^{*}}{2}d\mu=0.\]
Since $s^{*}=s+|\nabla\omega|^{2}\geq 0$, we get $s^{*}=s=0$. 
Therefore, $(M, g, \omega)$ is scalar-flat K\"ahler.

\end{proof}

By definition, a riemannian manifold $(M, g)$ is conformally flat if for any point, there is a neighborhood $U$ 
and a smooth function such that $(U, e^{2f}g)$ is flat. 
When the dimension of a riemannian manifold is greater than or equal to 4, 
then $(M, g)$ is conformally flat if and only if the Weyl tensor vanishes [2].

\newtheorem{Theorem}{Theorem}
\begin{Theorem}
(Tanno, [23])
Let $(M, g, \omega)$ be a conformally flat K\"ahler surface. 
Then $(M, g)$ is either locally flat or locally a product space of 2-dimensional riemannian manifolds of constant curvature
$K$ and $-K$. 

\end{Theorem}

\begin{Theorem}
Let $(M, g, \omega)$ be a compact, self-dual almost-K\"ahler 4-manifold with zero scalar curvature. 
Then $(M, g, \omega)$ is conformally flat K\"ahler surface and therefore, 
$(M, g)$ is either locally flat or locally a product space of 2-dimensional riemannian manifolds of constant
curvature $K$ and $-K$.
\end{Theorem}
\begin{proof}
We use Proposition 1. 
If $(M, g, \omega)$ is self-dual and scalar-flat K\"ahler, 
then the result follows from Theorem 1. 

Suppose $M$ is diffeomorphic to a rational or ruled surface. 
If $b_{-}=0$, then $M$ is diffeomorphic to $\mathbb{CP}_{2}$. 
From the following formula, 
\[2\chi-3\tau=\frac{1}{4\pi^{2}}\int_{M}\left(\frac{s^{2}}{24}+2|W_{-}|^{2}-\frac{|ric_{0}|^{2}}{2}\right)d\mu,\]
we get $2\chi-3\tau\leq 0$ for a self-dual metric with zero scalar curvature. 
However, for $\mathbb{CP}_{2}$, $2\chi-3\tau>0$.
Thus, a manifold diffeomorphic to $\mathbb{CP}_{2}$ does not admit a self-dual metric with zero scalar curvature.

Suppose $b_{-}\neq 0$.
Since $g$ is self-dual, we get $b_{-}=1$ and thus, $g$ is anti-self-dual. 
Then $(M, g, \omega)$ is self-dual and anti-self-dual almost-K\"ahler with zero scalar curvature. 
On the other hand, from the following formula for an almost-K\"ahler four-manifold, 
\[s^{*}\equiv\frac{s}{3}+2W_{+}(\omega, \omega)=s+|\nabla\omega|^{2}.\]
we get $s^{*}=\frac{s}{3}$ for an anti-self-dual almost-K\"ahler four manifold. 
Since $s=0$, we have $s^{*}=0$ and $\nabla\omega=0$. 
Thus, $(M, g, \omega)$ is K\"ahler. 
Since $s=0$, $(M, g, \omega)$ is a conformally flat K\"ahler surface. 
\end{proof}

\vspace{50pt}

\section{\large\textbf{Almost-K\"ahler four manifolds with nonnegative biorthogonal curvature}}\label{S:Intro}
We consider a compact almost-K\"ahler four manifold with $K^{\perp}\geq 0$.
Since $W_{\pm}$ is symmetric, there exist eigenvectors $\alpha_{i}^{+}$ with eigenvalues $\lambda_{i}^{\pm}$, $1\leq i\leq 3$ at $p$. 
We assume $\lambda_{1}^{\pm}\leq \lambda_{2}^{\pm}\leq \lambda_{3}^{\pm}$.
We define two functions $\kappa^{\pm}:M^{4}\to \mathbb{R}$ by 
\[\kappa^{\pm}:=\frac{s}{3}-2\lambda_{3}^{\pm}.\]

\newtheorem{Lemma}{Lemma}
\begin{Lemma}
Let $(M, g)$ be a compact, smooth oriented riemannian four manifold and there exists a self-dual harmonic 2-form of constant length. 
Then $\kappa^{+}\leq 0$. 
\end{Lemma}
\begin{proof}
Let $\alpha$ be a self-dual harmonic 2-form of constant length. 
From the Weitzenb\"ock formula, 
\[\Delta\alpha=\nabla^{*}\nabla\alpha-2W_{+}(\alpha, \cdot)+\frac{s}{3}\alpha,\]
we get
\[0=\frac{1}{2}\Delta|\alpha|^{2}+|\nabla\alpha|^{2}-2W_{+}(\alpha, \alpha)+\frac{s}{3}|\alpha|^{2}.\]
From this, we get
\[\frac{s}{3}|\alpha|^{2}-2W_{+}(\alpha, \alpha)\leq 0.\]
Let $\alpha=a\alpha_{1}^{+}+a_{2}\alpha_{2}^{+}+a_{3}\alpha_{3}^{+}$, 
where $\alpha_{i}^{+}$ are eigenvectors of $W_{+}$ with eigenvalues $\lambda_{i}^{+}$. 
Let $c:=a_{1}^{2}+a_{2}^{2}+a_{3}^{2}$ be a positive constant. 
Then we have 
\[0\geq\frac{s}{3}|\alpha|^{2}-2W_{+}(\alpha, \alpha)=\frac{s}{3}(a_{1}^{2}+a_{2}^{2}+a_{3}^{2})
-2(\lambda_{1}^{+}a_{1}^{2}+\lambda_{2}^{+}a_{2}^{2}+\lambda_{3}^{+}a_{3}^{2})\geq \left(\frac{s}{3}-2\lambda_{3}^{+}\right)c=\kappa^{+}c.\]

\end{proof}

\begin{Lemma}
Let $(M, g)$ be a compact, smooth oriented riemannian manifold  with nonnegative biorthogonal curvature. 
Then we have $\kappa^{+}+\kappa^{-}\geq 0$. Equivalently, $K^{\perp}\leq\frac{s}{4}$.
Moreover, the scalar curvature is nonnegative.

\end{Lemma}
\begin{proof}

Let $P$ be a 2-plane and $K$ be a sectional curvature. Then we have [3], [11] 
\[2K_{1}^{\perp}=\frac{s}{6}+\lambda_{1}^{+}+\lambda_{1}^{-}\]
\[2K_{3}^{\perp}=\frac{s}{6}+\lambda_{3}^{+}+\lambda_{3}^{-},\]
where 
\[K_{1}^{\perp}=\min\limits_{P\in T_{p}M}\left(\frac{K(P)+K(P^{\perp})}{2}\right)\]
\[K_{3}^{\perp}=\max\limits_{P\in T_{p}M}\left(\frac{K(P)+K(P^{\perp})}{2}\right).\]
Since $W_{\pm}$ is trace-free, we have
$\lambda_{1}^{\pm}+ \lambda_{2}^{\pm}+ \lambda_{3}^{\pm}=0$ and therefore, 
\[-\frac{1}{2}\lambda_{1}^{\pm}\leq \lambda_{3}^{\pm}\leq -2\lambda_{1}^{\pm}.\]
Since $K^{\perp}\geq 0$, we have 
\[0\leq \frac{s}{6}+\lambda_{1}^{+}+\lambda_{1}^{-}\leq \frac{s}{6}-\frac{\lambda_{3}^{+}}{2}-\frac{\lambda_{3}^{-}}{2}.\]
From this, we get 
\[0\leq (\frac{s}{12}-\frac{\lambda_{3}^{+}}{2})+ (\frac{s}{12}-\frac{\lambda_{3}^{-}}{2}).\]
Then we have 
\[0\leq \frac{2s}{3}-2\lambda_{3}^{+}-2\lambda_{3}^{-}=\kappa^{+}+\kappa^{-}=s-4K_{3}^{\perp}\leq s-4K^{\perp}.\]
From $0\leq \frac{2s}{3}-2\lambda_{3}^{+}-2\lambda_{3}^{-}$ and $\lambda_{3}^{\pm}\geq 0$, we get $s$ is nonnegative. 

\end{proof}

\begin{Lemma}
Let $(M, g)$ be a compact, smooth oriented riemannian four manifold with nonnegative biorthogonal curvature. 
Suppose there exists a self-dual harmonic 2-form of constant length. 
Then $\frac{s}{6}-W_{-}\geq 0$. 
\end{Lemma}
\begin{proof}
From $\kappa^{+}\leq 0$ and $\kappa^{+}+\kappa^{-}\geq 0$, we get $\kappa^{-}\geq 0$.
Since $\left(\frac{s}{6}-W_{-}\right)\geq \kappa^{-}$, we get the result.

\end{proof}

\newtheorem{Corollary}{Corollary}
\begin{Corollary}
Let $(M, g, \omega)$ be a compact almost-K\"ahler four manifold with nonnegative biorthogonal curvature. 
Then $\frac{s}{6}-W_{-}\geq 0$.
\end{Corollary}
\begin{proof}
We note that $\omega$ is a self-dual harmonic 2-form of constant length. 
\end{proof}

\begin{Theorem}
Let $(M, g, \omega)$ be a compact almost-K\"ahler four manifold with nonnegative biorthogonal curvature.
Then 
\begin{itemize}
\item $M$ is diffeomorphic to $\mathbb{CP}_{2}$; or
\item $(M, g, \omega)$ is locally flat K\"ahler; or
\item $(M, g, \omega)$ is K\"ahler 
and $(M, g)$ is locally a product space of 2-dimensional riemannian manifolds. 
\end{itemize}

\end{Theorem}
\begin{proof}
Since $W_{-}$ is trace-free and $\frac{s}{6}-W_{-}\geq0$, we get $s$ is nonnegative. 
Suppose $s$ is not identically zero. 
When $M$ carries a metric with nonnegative scalar curvature which is not identically zero, 
$M$ admits a metric with constant positive scalar curvature and zero scalar curvature [2]. 
Therefore, $M$ admits a metric of positive scalar curvature. 
Then by [18], [21], $M$ is diffeomorphic to a rational or ruled surface.
If $b_{-}=0$, $M$ is diffeomorphic to $\mathbb{CP}_{2}$.

Suppose $b_{-}\neq 0$.
Then, there exists a non-trivial anti-self-dual harmonic 2-form $\alpha$ on $M$. 
From the Weitzenb\"ock formula, we have
\[\int_{M}\left(|\nabla\alpha|^{2}+\left(\frac{s}{3}-2W_{-}\right)(\alpha, \alpha)\right)d\mu=0.\]
Since $\frac{s}{6}-W_{-}\geq 0$, $\nabla\alpha=0$.
Then $(\overline{M}, g, \alpha)$ is K\"ahler with $K^{\perp}\geq 0$. 
Then by Lemma 3, we get $\frac{s}{6}-W_{-}\geq 0$ for $(\overline{M}, g)$. 
Since $(\overline{M}, g, \alpha)$ is K\"ahler and $\overline{M}$ admits a metric of positive scalar curvature, 
we get $\overline{M}$ is diffeomorphic to a rational or ruled surface. 
In particular, $b_{+}(\overline{M})=1$, which implies $b_{-}(M)=1$. 
From this, we get $M$ is diffeomorphic to $S^{2}$-bundle over $\Sigma_{g}$ for $g\geq 0$. 
By reversing the orientation, we get $\frac{s}{6}-W_{+}\geq 0$ for $(M, g)$. 
Then $\omega$ is parallel. 
Let $p\in M$ and 
\[V_{1}=\{w\in T_{p}M|\alpha(w, v)=0, \forall v\in T_{p}M\}\]
\[V_{2}=\{w\in T_{p}M|\omega(w, v)=0, \forall v\in T_{p}M\}.\]
Since $\left<\alpha, \omega\right>=0$, $V_{1}\perp V_{2}$. 
Let $P$ be a parallel transport at $p$. 
Since $P$ is an isometry, there exists nonzero $v'$ such that $v=P(v')$ for any nonzero $v\in T_{p}M$. 
Since $\alpha$ is parallel, we have [22] 
\[\alpha(Pw, v)=\alpha(Pw, Pv')=\alpha(w, v')=0. \]
Thus, $V_{1}$ is an invariant subspace. 
Similarly, $V_{2}$ is an invariant subspace. 
They by de Rham decomposition theorem, $g$ is locally a product metric.

Suppose $s\equiv 0$. Since $\frac{s}{6}-W_{-}\geq 0$, we get $W_{-}=0$. 
Then the result follows from Theorem 2. 
\end{proof}

\begin{Corollary}
Let $(M, g, \omega)$ be a compact almost-K\"ahler four manifold with $K\geq0$.
Then  
\begin{itemize}
\item $M$ is diffeomorphic to $\mathbb{CP}_{2}$; or
\item $(M, g, \omega)$ is locally flat K\"ahler; or
\item $(M, g, \omega)$ is K\"ahler 
and $(M, g)$ is locally a product space of $S^{2}$ and $S^{2}$ with metrics of nonnegative curvature on each factor; or
\item $(M, g, \omega)$ is K\"ahler 
and $(M, g)$ is locally a product space of $S^{2}$ and $T^{2}$ with a flat metric on $T^{2}$ and a metric of nonnegative curvature on $S^{2}$.

\end{itemize}
\end{Corollary}
\begin{proof}
We note that by Gauss-Bonnet theorem, a metric of nonnegative curvature on $T^{2}$ is flat. 

\end{proof}

\begin{Corollary}
Let $(M, g, \omega)$ be a compact almost-K\"ahler four manifold with $K^{\perp}\geq 0$. 
Suppose $(M, g)$ is not locally flat and the Ricci curvature is greater than or equal to zero. 
Then we have $0\leq c_{2}\leq c_{1}^{2}$. 
\end{Corollary}
\begin{proof}
From Theorem 3, we have $\tau\geq 0$. 
Since $c_{1}^{2}=2\chi+3\tau$, it is enough to show that $\chi\geq 0$.
Note that except $S^{2}$-bundle over $\Sigma_{g}$, $g\geq 2 $, $\chi\geq 0$.
Almost-K\"ahler metric with $K^{\perp}\geq 0$ on such a manifold is locally a product metric. 
On $\Sigma_{g}$, for $g\geq 2$, there is a point $P$ such that 
$K<0$. Let $\{e_{1}, e_{2}\}$ be an orthonormal basis on $\Sigma_{g}$ at $p$
and $\{e_{3}, e_{4}\}$ is an orthonormal basis on $S^{2}$. 
Then $K_{12}+K_{13}+K_{14}=K_{12}<0$. 
Thus, there exists a point where the Ricci curvature is less than zero. 
\end{proof}

\begin{Proposition}
Let $(M, g)$ be a compact, smooth oriented riemannian four manifold with nonnegative biorthogonal curvature. 
Suppose there exists a non-trivial harmonic 2-form $\alpha$ of constant length such that 
$|\alpha^{+}|\neq |\alpha^{-}|$ at any point, where $\alpha^{+}$ is the self-dual part of $\alpha$ and 
$\alpha^{-}$ is the anti-self-dual part of $\alpha$. 
Then 
\begin{itemize}
\item $\alpha$ is self-dual; or
\item $\alpha$ is anti-self-dual; or 
\item $\nabla\alpha^{+}=\nabla\alpha^{-}=0$, $\alpha^{+}\neq0$, $\alpha^{-}\neq 0$. 
\end{itemize}

\end{Proposition}
\begin{proof}
We note that $\alpha^{\pm}$ are harmonic. 
From the Weitzenb\"ock formula, we get 
\[\frac{s}{3}|\alpha|^{2}-2W(\alpha, \alpha)\leq 0.\]
From this, we get 
\[\kappa^{+}|\alpha^{+}|^{2}+\kappa^{-}|\alpha^{-}|^{2}\leq 0.\]
Suppose $|\alpha^{+}|>|\alpha^{-}|$. Then from, 
\[(\kappa^{+}+\kappa^{-})|\alpha^{-}|^{2}-\kappa^{-}(|\alpha^{+}|^{2}-|\alpha^{-}|^{2})=\kappa^{+}|\alpha^{+}|^{2}+\kappa^{-}|\alpha^{-}|^{2}\leq 0,\]
we get $\kappa^{-}\geq 0$. Thus, from the Weitenb\"ock formula for an anti-self-dual 2-form, we get $\nabla\alpha^{-}=0$. 
This gives $\kappa^{-}\leq 0$. 
Thus, we get $\kappa^{-}=0$. From $\kappa^{+}+\kappa^{-}\geq 0$, we get $\kappa^{+}\geq 0$. 
Then we have $\nabla\alpha^{+}=0$. When we have $|\alpha^{+}|<|\alpha^{-}|$, we get the same result. 

\end{proof}
Below we consider the isotropic curvature following [20]. 
Let $(M, g)$ be a smooth oriented riemannian four-manifold. 
Let $V$ be a complex linear subspace such that $V\in T_{p}M\otimes\mathbb{C}$. 
We denote $(, )$ a complex bilinear form on $T_{p}M\otimes\mathbb{C}$, which is induced by the riemannian metric. 
By $\ll, \gg$, we denote a Hermitian inner product on $T_{p}M\otimes\mathbb{C}$ induced by the riemannian metric.
Then, we have $\ll z, w\gg=(z, \overline{w})$ for $z, w\in T_{p}M\otimes\mathbb{C}$.
By definition, $V$ is totally isotropic if $(z, z)=0$ for $z\in V$. 

The curvature operator can be extended to a complex linear map 
\[\mathfrak{R}: \Lambda^{2}T_{p}M\otimes\mathbb{C}\to \Lambda^{2}T_{p}M\otimes\mathbb{C}.\]
A complex sectional curvature $\mathbb{K}(\sigma)$ for a two plane $\sigma \in T_{p}M\otimes\mathbb{C}$
is defined by 
\[\mathbb{K}(\sigma)=\frac{\ll \mathfrak{R}(z\wedge w), z\wedge w\gg}{||z\wedge w||^{2}}.\]
A riemannian manifold $(M, g)$ has nonnegative isotropic curvature if $\mathbb{K}(\sigma)\geq 0$
where $\sigma \in T_{p}M\otimes\mathbb{C}$ is a totally isotropic two-plane.

Let $\{e_{1}, e_{2}, e_{3}, e_{4}\}$ be a positively oriented orthonormal basis. 
Then for $z=e_{1}+ie_{2}, w=e_{3}+ie_{4}$, we have $(z, z)=(w, w)=0$. 
Then we have 
\[\ll \mathfrak{R}(z\wedge w), z\wedge w\gg=\left<\mathfrak{R}(e_{1}\wedge e_{3}+e_{4}\wedge e_{2}), e_{1}\wedge e_{3}+e_{4}\wedge e_{2}\right>\]
\[+\left<\mathfrak{R}(e_{1}\wedge e_{4}+e_{2}\wedge e_{3}), e_{1}\wedge e_{4}+e_{2}\wedge e_{3}\right>..\]
We note that $\phi_{1}^{+}:=e_{1}\wedge e_{2}+e_{3}\wedge e_{4}$,  $\phi_{2}^{+}=e_{1}\wedge e_{3}+e_{4}\wedge e_{2}$,  
 $\phi_{3}^{+}=e_{1}\wedge e_{4}+e_{2}\wedge e_{3}$ is an orthogonal basis of $\Lambda^{+}M$ with length $\sqrt{2}$. 
 We use the same notation for dual 1-forms of $e_{i}$. 
Using the curvature operator, 

\[ 
\mathfrak{R}=
\LARGE
\begin{pmatrix}

 \begin{array}{c|c}
\scriptscriptstyle{W_{+}\hspace{5pt}\scriptstyle{+}\hspace{5pt}\frac{s}{12}}I& \scriptscriptstyle{ric_{0}}\\
 \hline
 
 \hspace{10pt}
 
 \scriptscriptstyle{ric_{0}^{*}}& \scriptscriptstyle{W_{-}\hspace{5pt}\scriptstyle{+}\hspace{5pt}\frac{s}{12}}I\\
 \end{array}
 \end{pmatrix}
 \]
 we get
 \[\ll \mathfrak{R}(z\wedge w), z\wedge w\gg=\left(W_{+}+\frac{s}{12}I\right)(\phi_{2}^{+}, \phi_{2}^{+})+\left(W_{+}+\frac{s}{12}I\right)(\phi_{3}^{+}, \phi_{3}^{+}).\]
Since $W_{+}$ is trace-free, we have 
\[W_{+}(\phi_{1}^{+}, \phi_{1}^{+})+W_{+}(\phi_{2}^{+}, \phi_{2}^{+}) +W_{+}(\phi_{3}^{+}, \phi_{3}^{+})=0.\] 
Thus, we get 
 \[\ll \mathfrak{R}(z\wedge w), z\wedge w\gg=\left(\frac{s}{6}I-W_{+}\right)(\phi_{1}^{+}, \phi_{1}^{+}).\] 
 We note that for a non-trivial self-dual 2-form $\phi$, there exists a positively oriented basis 
 $\{e_{1}, e_{2}, e_{3}, e_{4}\}$ of $T_{p}M$ such that $\phi_{p}=c(e_{1}\wedge e_{2}+e_{3}\wedge e_{4})$ for $c>0$. 
 Thus, we have $\frac{s}{6}-W_{+}\geq 0$ if $(M, g)$ has nonnegative isotropic curvature. 
 
 We note that when $\{e_{1}, e_{2}, e_{3}, e_{4}\}$ is a positively oriented basis, 
 $\{-e_{1}, e_{2}, e_{3}, e_{4}\}$ is negatively oriented. 
 Let $z=-e_{1}+ie_{2}$, $w=e_{3}+ie_{4}$. 
 Then $(z, z)=(w, w)=0$. 
 Thus, we get 
 \[\ll \mathfrak{R}(z\wedge w), z\wedge w\gg=\left(\frac{s}{6}-W_{-}\right)(\phi_{1}^{-}, \phi_{1}^{-}),\] 
 where $\phi_{1}^{-}=e_{1}\wedge e_{2}-e_{3}\wedge e_{4}$. 
 For an anti-self-dual 2-form $\phi$, there exists a positively oriented basis $\{e_{1}, e_{2}, e_{3}, e_{4}\}$
 such that $\phi_{p}=c(e_{1}\wedge e_{2}-e_{3}\wedge e_{4})$ for $c>0$. 
 Thus, we get $\frac{s}{6}-W_{-}\geq 0$ if $(M, g)$ has nonnegative isotropic curvature.

\begin{Proposition}
Let $(M, g, \omega, J)$ be a compact almost-K\"ahler four manifold with nonnegative isotropic curvature.
Then 
\begin{itemize}
\item $(M, J)$ is biholomorphic to $\mathbb{CP}_{2}$; or
\item $(M, g, \omega)$ is locally flat K\"ahler; or 
\item $(M, g, \omega)$ is K\"ahler and $(M, g)$ is locally a product metric of 2-dimensional riemannian manifolds. 
\end{itemize}

\end{Proposition}
\begin{proof}
We have $\frac{s}{6}-W_{+}\geq0$ and $\frac{s}{6}-W_{-}\geq0$.
Note that $\omega$ is a self-dual harmonic 2-form of length $\sqrt{2}$. 
From the Weitzenb\"ock formula, 
\[0=\int_{M}\left(|\nabla\omega|^{2}-2W_{+}(\omega, \omega)+\frac{s}{3}|\omega|^{2}\right)d\mu,\]

we get $\nabla\omega=0$. Thus, $(M, g, \omega, J)$ is K\"ahler. 
When $b_{-}\geq 1$, there is a parallel anti-self-dual 2-form. 
We have either $M$ admits a metric of positive scalar curvature, which implies $b_{\pm}=1$ or $s=W_{\pm}=0$. 
We note that if a complex surface $(M, J)$ is diffeomorphic to $\mathbb{CP}_{2}$, 
then $(M, J)$ is biholomorphic to $\mathbb{CP}_{2}$ [25]. 

\end{proof}

\vspace{50pt}

\section{\large\textbf{Almost-K\"ahler four manifolds with nonnegative biorthogonal curvature and constant scalar curvature}}\label{S:Introduction}

We consider Gray's Theorem [8] in compact almost-K\"ahler four manifolds with $J$-invariant Ricci tensor. 

\begin{Theorem}
(Gray) Let $(M, g, \omega)$ be a compact K\"ahler manifold with nonnegative sectional curvature. 
Suppose $M$ has constant scalar curvature. 
Then $(M, g)$ is locally symmetric. 
If $(M, g, \omega)$ has positive sectional curvature, then $(M, g, \omega)$ is $\mathbb{CP}_{n}$ with the Funibi-Study metric. 

\end{Theorem}

\begin{Theorem}
(Dr$\breve{a}$ghici) Let $(M, g, \omega)$ be a compact, almost-K\"ahler four manifold with $J$-invariant Ricci tensor. 
Let $\rho(X, Y)=Ric(JX, Y)$ be the corresponding 2-form.
Then $\rho$ is closed. 
\end{Theorem}
\begin{proof}
We follow the proof given in [4], [5]. 
Let $\delta$ be the adjoint to $d$ so that 
\[\int_{M}\left<\delta \omega_{1}, \omega_{2}\right>d\mu=\int_{M}\left<\omega_{1}, d\omega_{2}\right>d\mu, \]
where $\omega_{1}$ is a 2-form and $\omega_{2}$ is a 1-form. 
Then $\delta=*d*$ and 
\[(\delta\alpha)(X)=-(\nabla_{E_{i}}\alpha)(E_{i}, X), \]
where $X\in T_{p}M$ and $\alpha$ is a 2-form [22]. 
Let $S$ be a $(0, 2)$-tensor. 
Then $div S$ is a $(0, 1)$-tensor defined by 
\[(divS)(X)=(\nabla_{E_{i}}S)(E_{i}, X).\]
We note that $div(Ric-\frac{s}{2}g)=0$ by the Second Bianchi Identity. We show 
\[(\delta\rho)(X)=divRic(JX).\]
It is enough to check using a basis $\{e_{1}, Je_{1}, e_{3}, Je_{3}\}$. 

The Nijenhuis tensor is given by 
\[N(X, Y)=[JX, JY]-[X, Y]-J[JX, Y]-J[X, JY].\]
Then by propsition 4.2 in Ch. IX [13], we have 
\[2\left<(\nabla_{X}J)Y, Z\right>=\left(N(Y, Z), JX\right>.\]
Since $N(J\cdot, \cdot)=-JN(\cdot, \cdot)$, we have 
\[2\left<(\nabla_{JX}J)JY, Z\right>=\left<N(JY, Z), -X\right>=-\left<JN(Y, Z), -X\right>=-\left(N(Y, Z), JX\right>\]
From this, we get 
\[(\nabla_{JX}J)JY=-(\nabla_{X}J)Y.\]
Using this formula and $J$-invariance of Ricci tensor, we get 
\[-(\nabla_{e_{i}}\rho)(e_{i}, X)-(\nabla_{Je_{i}}\rho)(Je_{i}, X)=(\nabla_{e_{i}}Ric)(e_{i}, JX)+(\nabla_{Je_{i}}Ric)(Je_{i}, JX),\]
for $i=1, 3$. Similarly, we can check $(\delta(\frac{s}{2}\omega))(X)=(div(\frac{s}{2}g))(JX)$. 
We note that 
\[\rho(JX, JY)=Ric(-X, JY)=\rho(X, Y).\]
Thus, $\rho$ is a $(1, 1)$-form. Since $\Lambda_{0}^{1, 1}$, which are $(1, 1)$-forms orthogonal to $\omega$,
is the set of anti-self-dual 2-forms, we get 
$\rho-\frac{\left<\rho, \omega\right>}{||\omega||^{2}}\omega$ is an anti-self-dual 2-form. We have
\[\left<\rho, \omega\right>=Ric(Je_{1}, Je_{1})+Ric(Je_{3}, Je_{3}).\]
Since $Ric$ is $J$-invariant, we have $Ric(Je_{1}, Je_{1})=Ric(e_{1}, e_{1})$ and $Ric(Je_{3}, Je_{3})=Ric(e_{3}, e_{3})$. 
Thus, we have $\left<\rho, \omega\right>=\frac{1}{2}s$ and 
\[\rho-\frac{\left<\rho, \omega\right>}{||\omega||^{2}}\omega=\rho-\frac{s}{4}\omega.\]
Since $\delta(\rho-\frac{s}{2}\omega)=0$, we get
\[0=*\delta\left(\rho-\frac{s}{2}\omega\right)=*^{2}d*\left(\rho-\frac{s}{2}\omega\right)=-d*\left(\rho-\frac{s}{2}\omega\right).\]
Note that $\rho-\frac{s}{4}\omega$ is an anti-self-dual 2-form and $\omega$ is a self-dual 2-form. 
Then we get 
\[0=d*\left(\rho-\frac{s}{4}\omega-\frac{s}{4}\omega\right)=d\left(-\rho+\frac{s}{4}\omega\right)-d\left(\frac{s}{4}\omega\right)=-d\rho.\]

\end{proof}

\begin{Theorem}
Let $(M, g, \omega)$ be a compact almost-K\"ahler four manifold with $J$-invariant Ricci tensor. 
Suppose the scalar curvature is constant and $b_{-}=0$. 
Then $(M, g, \omega)$ is self-dual K\"ahler-Einstein and $b_{1}=0$. 
\end{Theorem}
\begin{proof}
When $s$ is constant, we note that $\rho_{0}=\rho-\frac{s}{4}\omega$ is an anti-self-dual closed 2-form. 
Since $\delta=*d*$ on 2-forms on compact four manifolds, we have 
\[\delta\left(\rho-\frac{s}{4}\omega\right)=*d*\left(\rho-\frac{s}{4}\omega\right)=-*d\left(\rho-\frac{s}{4}\omega\right)=0.\]
Thus, $\rho_{0}=\rho-\frac{s}{4}\omega$ is an anti-self-dual harmonic 2-form. 
Since $b_{-}=0$, we get $\rho_{0}=0$. 
Using $J$-invariance of Ricci tensor, we can check $\rho_{0}=0$ implies $(M, g)$ is Einstein. 
A compact almost-K\"ahler four manifold with $\delta W_{+}=0$ has 
\[\int_{M}\frac{s^{2}}{24}d\mu\geq \int_{M}|W_{+}|^{2}d\mu,\]
with equality if and only if $(M, g, \omega)$ is K\"ahler [12], [16]. Oriented, Einstein four manifolds have $\delta W_{+}=0$. 
Then for a compact almost-K\"ahler Einstein manifold, we have 
\[\chi-3\tau=\frac{1}{8\pi^{2}}\int_{M}\left(\frac{s^{2}}{24}-|W_{+}|^{2}+3|W_{-}|^{2}-\frac{|ric_{0}|^{2}}{2}\right)d\mu\geq0.\]
Since $\chi=2-b_{1}+b_{+}-b_{3}$ and $\tau=b_{+}$, we get 
\[\chi-3\tau=2-b_{1}+b_{+}-b_{3}-3b_{+}.\]
Since $b_{+}\geq 1$, we have $b_{1}=b_{3}=0$ and $b_{+}=1$. 
Thus, we have $\chi=3\tau$. 
Then we have $\int_{M}\frac{s^{2}}{24}d\mu= \int_{M}|W_{+}|^{2}d\mu$ and $W_{-}=0$. 
Thus, $(M, g, \omega)$ is self-dual K\"ahler-Einstein. 
\end{proof}

\begin{Corollary}
Let $(M, g, \omega)$ be a compact, almost-K\"ahler four manifold with $J$-invariant Ricci tensor. 
Suppose $(M, g)$ has positive biorthogonal curvature and constant scalar curvature. 
Then $(M, g, \omega)$ is $\mathbb{CP}_{2}$ with the Fubini-Study metric up to rescaling. 
\end{Corollary}
\begin{proof}
We note that $\frac{s}{6}-W_{-}>0$ and $s>0$. 
Then from a Weitzenb\"ock formula for an anti-self-dual 2-form,
\[\Delta\alpha=\nabla^{*}\nabla\alpha-2W_{-}(\alpha, \cdot)+\frac{s}{3}\alpha,\]
we get $b_{-}=0$. 
\end{proof}

\begin{Proposition}
Let $(M, g, \omega)$ is a compact almost-K\"ahler four manifold with $J$-invariant Ricci tensor. 
Suppose $(M, g)$ has nonnegative biorthogonal curvature and the scalar curvature is constant. 
Then either $(M, g)$ is Einstein, or $(M, g, \omega)$ is K\"ahler with parallel Ricci form. 
\end{Proposition}
\begin{proof}
When $s$ is constant, we note that $\rho_{0}=\rho-\frac{s}{4}\omega$ is an anti-self-dual harmonic 2-form. 
Since $\frac{s}{6}-W_{-}\geq 0$, from the Weitzenb\"ock formula for an anti-self-dual 2-form, 
we get $\nabla\rho_{0}=0$.
Thus, if $(M, g, \omega)$ is not Einstein, we get a parallel anti-self-dual 2-form $\rho_{0}$. 
Then $(\overline{M}, g, \rho_{0})$ is a K\"ahler surface with $K^{\perp}\geq0$. Then by Corollary 1, we get $\nabla\omega=0$. 
Then $\nabla\rho=\nabla\left(\rho_{0}-\frac{s}{4}\omega\right)=0$.
\end{proof}

\begin{Proposition}
Let $(M, g, \omega)$ be a compact almost-K\"ahler four-manifold with nonnegative biorthogonal curvature. 
Suppose the scalar curvature is constant. Then either $M$ is diffeomorphic to $\mathbb{CP}_{2}$ or
$(M, g, \omega)$ is K\"ahler and locally symmetric. 
\end{Proposition}
\begin{proof}
We note that the scalar curvature is nonnegative. 
Suppose the scalar curvature is positive constant. 
Then $M$ is diffeomrophic to a rational or ruled surface [18], [21]. 
If $M$ is not diffeomorphic to $\mathbb{CP}_{2}$, then $b_{-}\neq 0$. 
Thus, there is a nontrivial anti-self-dual harmonic 2-form $\alpha$. 
Since $\frac{s}{6}-W_{-}\geq 0$, $\alpha$ is parallel. 
Therefore, $(\overline{M}, g, \alpha)$ is K\"ahler. 
Then, $(\overline{M}, g, \alpha)$ is K\"ahler with nonnegative biorthogonal curvature. 
Thus, we get $\frac{s}{6}-W_{+}\geq 0$, which implies $\omega$ is parallel. 
Since $(M, g, \omega)$ and $(\overline{M}, g, \alpha)$ are K\"ahler, we get $W_{\pm}$ is constant since $s$ is constant. 
Since $(M, g, \omega)$ is K\"ahler, $\nabla\rho=0$ if $(M, g, \omega)$ is not Einstein by Proposition 4. 
Thus, $(M, g, \omega)$ is locally symmetric. 
Suppose $s=0$. 
Then from $\frac{s}{6}-W_{-}\geq0$, we get $W_{-}=0$. 
Thus, $(M, g, \omega)$ is self-dual. The the result follows from Theorem 2. 

\end{proof}

\begin{Corollary}
Let $(M, g, \omega)$ be a compact almost-K\"ahler four manifold with $J$-invariant Ricci tensor. 
Suppose $(M, g)$ has nonnegative biorthogonal curvature and constant scalar curvature. 
Then $(M, g, \omega)$ is K\"ahler and locally symmetric. 
\end{Corollary}
\begin{proof}
The result follows from Theorem 6 and Proposition 5. 
\end{proof}

Blair connection on an almost-K\"ahler four manifold is given by 
\[\tilde{\nabla}_{X}Y=\nabla_{X}Y-\frac{1}{2}J(\nabla_{X}J)Y, \]
where $\nabla$ is Levi-Civita connection. 
Then $\tilde{\nabla}$ induces the connection on the anti-canonical line bundle and its curvature is given in [14].

\begin{Proposition}
Let $(M, g, \omega)$ be a compact almost-K\"ahler four manifold with nonnegative biorthogonal curvature.
Suppose $s^{\tilde{\nabla}}:=\frac{s+s^{*}}{2}$ is constant. Then 
\begin{itemize}
\item $(M, g)$ is locally flat K\"ahler; or
\item $M$ is diffeomorphic to $\mathbb{CP}_{2}$ and 
$(M, g, \omega)$ is an almost-K\"ahler Hermitian-Einstein metric with positive $s^{\tilde{\nabla}}$; or
\item $(M, g, \omega)$ is K\"ahler and locally symmetric.
\end{itemize}
\end{Proposition}
\begin{proof}
We use Theorem 3. We note that $s^{\tilde{\nabla}}\geq 0$ since $s\geq 0$. 
Suppose $M$ is diffeomorphic to $\mathbb{CP}_{2}$. 
We note that if $s^{\tilde{\nabla}}=0$, then $c_{1}\cdot\omega=0$. 
On the other hand, it was shown [15] that if $M$ is a smooth manifold which is diffeomorphic to a rational or ruled surface with $2\chi+3\tau\geq 0$, 
then $c_{1}\cdot\omega>0$ for any symplectic form $\omega$ on $M$. 
Thus, $s^{\tilde{\nabla}}$ is positive. 
Since $\alpha^{+}=\frac{\alpha+*\alpha}{2}$, 
$\alpha$ is harmonic if $d\alpha^{+}=0$ and $d\alpha=0$. 
Since $b_{+}(\mathbb{CP}_{2})=1$, from the curvature formula of Blair connection [14], we have
\[(iF_{\tilde{\nabla}})^{+}=\langle W_{+}(\omega)^{\perp}+\left(\frac{s+s^{*}}{8}\right)\omega, \omega\rangle\omega=\frac{s^{\tilde{\nabla}}}{2}\omega.\]
Thus,  $iF_{\tilde{\nabla}}$ is harmonic if $s^{\tilde{\nabla}}$ is constant. 
Since $b_{+}(\mathbb{CP}_{2})=1$, we have $iF_{\tilde{\nabla}}=c\omega$
for a constant $c$. 
Thus, $(M, g, \omega)$ is an almost-K\"ahler Hermitian-Einstein four manifold, for which 
we refer to [17]. 

Suppose $(M, g, \omega)$ is K\"ahler and $M$ is not diffeomorphic to $\mathbb{CP}_{2}$. 
Since this metric is K\"ahler, we have $s=s^{*}$. Thus, $s$ is constant. 
Then the result follows from Proposition 5.

\end{proof}

\vspace{50pt}

\section{\large\textbf{K\"ahler surfaces with nonnegative orthogonal holomorphic bisectional curvature}}\label{S:Intro}
Let $(M, g, \omega)$ be a K\"ahler surface. 
The Holomorphic sectional curvature is defined by 
\[H(X)=R(X\wedge JX, X\wedge JX),\]
for a unit tangent vector $X\in TM$. 
The Holomorphic bisectional curvature is defined by 
\[H(X, Y)=R(X\wedge JX, ,Y\wedge JY), \]
for unit tangent vectors $X, Y\in TM$. 
Using the Second Bianchi Identity and K\"ahler property, we have 
\[H(X, Y)=R(X\wedge Y, X\wedge Y)+R(X\wedge JY, X\wedge JY).\]
The orthogonal holomorphic bisectional curvature is defined by 
\[H(X, Y)=R(X\wedge JX, Y\wedge JY),\]
for unit tangent vectors $X, Y$ such that $X\perp Y$ and $X\perp JY$.

\begin{Proposition}
Let $(M, g, \omega)$ be a compact connected K\"ahler manifold with 
positive orthogonal bisectional holomorphic curvature. 
Let $V$ and $W$ be compact complex submanifolds such that $dim V+dim W\geq dim M$. 
Then $V$ and $W$ have a nonempty intersection. 
\end{Proposition}
\begin{proof}
We follow the proof in [6], [7]. 
If $V\cap W=\phi$, let $\alpha(t)$ $0\leq t\leq r$ be a shortest geodesic from $V$ to $W$
with $p=\alpha(0), q=\alpha(r)$. 
Since $dim V +dim W\geq dim M$, there exists a parallel vector field $X$ which is tangent to $V$ at $p$ and $W$ at $q$. 
Since $\nabla J=0$ and $V, W$ are complex submanifolds, we get 
$JX$ is also parallel vector field which is tangent to $V$ at $p$ and $W$ at $q$.
Let $T$ be a vector field tangent to $\alpha(t)$ 
We note that $T\perp X$ and $T\perp JX$. 
Then the second variation of $\alpha(t)$ with respect to $X$ and $JX$ are given as follows. 
\[L_{X}^{''}(0)=g(\nabla_{X}X, T)_{q}-g(\nabla_{X}X, T)_{p}-\int_{0}^{r}R(T\wedge X, T\wedge X)dt\]
\[L_{JX}^{''}(0)=g(\nabla_{JX}JX, T)_{q}-g(\nabla_{JX}JX, T)_{p}-\int_{0}^{r}R(T\wedge JX, T\wedge JX)dt\]
It was shown in [6] 
\[g(\nabla_{X}X, T)_{p}+g(\nabla_{JX}JX, T)_{p}=0.\]
\[g(\nabla_{X}X, T)_{q}+g(\nabla_{JX}JX, T)_{q}=0.\]
Then we have 
\[L_{X}^{''}(0)+L_{JX}^{''}(0)=-\int_{0}^{r}R(T\wedge X, T\wedge X)+R(T\wedge JX, T\wedge JX)dt\]
\[=-\int_{0}^{r}R(T\wedge JT, X\wedge JX)dt<0.\]
We note that positive orthogonal bisectional curvature implies $R(T\wedge JT, X\wedge JX)>0$
since $T\perp X, T\perp JX$. 
Then we get a contradiction since $\alpha$ is a shortest geodesic from $V$ to $W$. 

\end{proof}

\begin{Proposition}
Let $(M, g, \omega)$ be a K\"ahler surface. 
If $(M, g, \omega)$ has nonnegative orthogonal holomorphic bisectional curvature, 
then $\frac{s}{6}-W_{-}\geq 0$. 
\end{Proposition}
\begin{proof}
We follow the proof given in [1]. 
Let $\phi$ be an anti-self-dual 2-form. 
Then an almost-complex structure is defined by 
\[g(KX, Y)=\frac{\sqrt{2}}{||\phi||}\phi(X, Y).\]
Similarly, $J$ correspond to a self-dual 2-form $\psi$. 
Since $\left<\phi, \psi\right>=0$, we have $\left<K, J\right>=0$.
Then we have 
\[\Sigma_{i=1}^{4}\left<K(e_{i}), J(e_{i})\right>=0.\]
If $\left<K(e_{i}), J(e_{i})\right>\neq 0$ for all $i$, then we can suppose
$\left<K(e_{1}), J(e_{1})\right>>0$ and $\left<K(e_{2}), J(e_{2})\right><0$.
Then there exists $t\in(0, 1)$ such that $\left<K(te_{1}+(1-t)e_{2}), J(te_{1}+(1-t)(e_{2})\right>=0$
Thus, there exists $V\neq 0$ such that $\left<K(V), J(V)\right>=0$. 
We note that from this, we get $\left<V, KJ(V)\right>=0$.
We consider $\{V, K(V), J(V), KJ(V)\}$, which is negatively oriented and $\{V, J(V), K(V), JK(V)\}$, which is positively oriented. 
Then we get $KJ(V)=JK(V)$. 
Then we have
\[J(V-KJV)=JV+KV=K(V-KJV).\]
Since $V\perp KJV$, we get $0\neq E_{1}:=V-KJV$ such that $JE_{1}=KE_{1}$. 
We also get 
\[J(V+KJV)=JV-KV=K(-V-KJV).\]
Thus, we have $0\neq E_{3}:=V+KJV$ such that $JE_{3}=-KE_{3}$. 
Since
$\left<V-KJV, V+KJV\right>=0$, we have $E_{1}\perp E_{3}$. 
Moreover, we have $\left<JE_{1}, E_{3}\right>=\left<JV-JKJV, V+KJV\right>=0$. 
We normalize $E_{1}$ and $E_{3}$ and we use the same notation for the dual 1-forms. 
Then, we get 
\[\phi=E_{1}\wedge KE_{1}+E_{3}\wedge KE_{3}=E_{1}\wedge JE_{1}-E_{3}\wedge JE_{3}.\]
We consider $E_{1}, E_{2}=JE_{1}, E_{3}, E_{4}=JE_{3}$. 
Then we have 
\[\left(\frac{s}{6}-W_{-}\right)(\phi, \phi)=R_{1313}+R_{1414}+R_{2323}+R_{2424}+2R_{1234},\]
where $R_{ijkl}=R(E_{i}\wedge E_{j}, E_{k}\wedge E_{l})$. 
By the curvature identities for K\"ahler manifolds, we have 
\[R_{1313}=R_{1324}, \hspace{5pt} R_{1414}=-R_{1423}, \hspace{5pt}  R_{2323}=-R_{2314}, \hspace{5pt} R_{2424}=R_{2413}.\]
From the First Bianchi Identity, we have 
\[R_{1243}+R_{3142}+R_{2341}=0, \hspace{5pt} R_{2413}+R_{3214}+R_{4312}=0.\]
From this, we get 
\[\left(\frac{s}{6}-W_{-}\right)(\phi, \phi)=R_{1313}+R_{1414}+R_{2323}+R_{2424}+2R_{1234}=4R_{1234}.\]
We note that $R_{1234}=R(E_{1}\wedge J E_{1}, E_{3}\wedge JE_{3})=H(E_{1}, E_{3})\geq 0$ and $E_{1}\perp E_{3}$, $E_{1}\perp JE_{3}$.  
\end{proof}

\begin{Theorem}
Let $(M, g, \omega, J)$ be a compact K\"ahler surface with nonnegative
orthogonal holomorphic bisectional curvature. 
Then 
\begin{itemize}
\item $(M, J)$ is biholomorphic to $\mathbb{CP}_{2}$; or
\item $(M, g)$ is locally flat; or 
\item $(M, g)$ is locally a product space  of 2-dimensional riemannian manifolds.   
\end{itemize}
\end{Theorem}
\begin{proof}
Since $\frac{s}{6}-W_{-}\geq 0$ and $W_{-}$ is trace-free, we have $s\geq 0$. 
Suppose $s$ is not identically zero. Then $M$ admits a positive scalar curvature metric [2]. 
If a compact symplectic four-manifold admits a metric of positive scalar curvature, 
then $M$ is diffeomorphic to $\mathbb{CP}_{2}$ if $b_{-}=0$. 
If $b_{-}\neq 0$, then from the Weitzenb\"ock formula, we get a parallel anti-self-dual 2-form. 
Then we get $(\overline{M}, g)$ is K\"ahler and $M$ admits a positive scalar curvature metric. 
Then $\overline{M}$ is diffeomorphic to a rational or ruled surface. 
Then $b_{+}(\overline{M})=1$. 
Thus, we get $b_{-}(M)=1$. 
Then we use the same argument in Theorem 3. 
We note that if a complex surface $(M, J)$ is diffeomorphic to $\mathbb{CP}_{2}$, 
then $(M, J)$ is biholomorphic to $\mathbb{CP}_{2}$ [25]. 

Suppose $s$ is identically zero. Since $\frac{s}{6}-W_{-}\geq 0$ and $W_{-}$ is trace-free, we have $W_{-}=0$.
For a K\"ahler surface, we have $W_{+}=0$ if and only if $s=0$. 
Thus, we get a conformally flat K\"ahler surface and we can use Theorem 1. 
\end{proof}

\vspace{20pt}

\begin{Corollary}
(Howard-Smyth) 
Let $(M, g, \omega, J)$ be a compact K\"ahler surface with nonnegative holomorphic bisectional curvature. 
Then 
\begin{itemize}
\item $(M, J)$ is biholomorphic to $\mathbb{CP}_{2}$; or
\item $(M, g)$ is locally flat; or
\item $(M, J)$ is biholomorphic to $\mathbb{CP}_{1}\times\mathbb{CP}_{1}$ 
and $(M, g)$ is a product metrics of nonnegative curvature; or 
\item $(M, J)$ is a ruled surface over an elliptic curve 
and $(M, g)$ is locally a product space with a flat metric on $T^{2}$ and a metric of nonnegative curvature on $S^{2}$. 

\end{itemize}
\end{Corollary}

\begin{proof}
Suppose $(M, g)$ is locally a product space of 2-dimensional riemannian manifolds. 
Then by classification of complex surfaces, $(M, J)$ is a minimal rational or ruled surface. 
We note that nonnegative holomorphic bisectional curvature implies nonnegative holomorphic sectional curvature. 
We also note that the tangent plane on the each factor on $S^{2}$-bundle over $\Sigma_{g}$ is holomorphic. Thus, $g\leq 1$. 
Among $\mathbb{CP}_{1}$-bundle over $\mathbb{CP}_{1}$, it was shown in [9] that 
only $\mathbb{CP}_{1}\times\mathbb{CP}_{1}$ has a metric of nonnegative holomorphic bisectional curvature.

\end{proof}

\vspace{20pt}
$\mathbf{Acknowledgments}$: The author would like to thank Prof. Claude LeBrun for helpful comments
and for suggesting the article before [11].


\newpage
\begin{thebibliography}{9}

\bibitem{AD}
V. Apostolov and J. Davidov,
\emph{Compact Hermitian surfaces and isotropic curvature},
Illinois Jour. Math. Vol 44, Number 2, summer 2000, 438-451.

\bibitem{B}
A. Besse, 
\emph{Einstein Manifolds}, 
Springer-Verlag, Berlin, 1987.


\bibitem{CR}
E. Costa and E. Ribeiro Jr., 
\emph{Four-dimensional compact manifolds with nonnegative biorthogonal curvature},
Michigan Math. J. 63 (2014), 747-761.



\bibitem{D1}
T. Dr$\breve{a}$ghici, 
\emph{On some 4-dimensional almost K\"ahler manifolds}, 
Kodai Math. J. 18 (1995), 156-168. 

\bibitem{D}
T. Dr$\breve{a}$ghici, 
\emph{Symplectic obstructions to the existence of $\omega$-compatible Einstein metrics}, 
Differential Geom. Appl. 22 (2005) 147-158. 




\bibitem{F}
T. Frankel, 
\emph{Manifolds with positive curvature}, 
Pacific J. Math. 11 (1961) 165-174. 



\bibitem{GK}
S. I. Goldberg and S. Kobayashi, 
\emph{Holomorphic bisectional curvature}, 
J. Differential Geom., 1 (1967) 225-233.

\bibitem{G}
A. Gray, 
\emph{Compact K\"ahler manifolds with nonnegative sectional curvature}, 
Invent. math. 41, 33-43 (1977). 



\bibitem{HS}
A. Howard and B. Smyth, 
\emph{K\"ahler surfaces of nonnegative curvature}, 
J. Differential Geom., 5 (1971) 491-502.

\bibitem{KK}
M. Kalafat and C. Koca, 
\emph{Conformally K\"ahler surfaces and orthogonal holomorphic bisectional curvature},
Geom. Dedicata 174, 401-408 (2015).



\bibitem{K}
I. Kim, 
\emph{Four-manifolds with harmonic 2-forms of constant length},
Geom. Dedicata 207, 209-218 (2020). 

\bibitem{K2}
I. Kim, 
\emph{Almost-K\"ahler four manifolds with harmonic self-dual Weyl curvature},
preprint.

\bibitem{KN}
S. Kobayashi and K. Nomizu, 
\emph{Foundations of Differential Geometry, $\Pi$}, 
Interscience, 1963. 

\bibitem{L}
C. LeBrun, 
\emph{Einstein metrics, Symplectic minimality, and Pseudo-holomorphic curves}, 
Ann. Glob. An. Geom. 28 (2005), 157-177.

\bibitem{L}
C. LeBrun, 
\emph{Weyl Curvature, Del Pezzo Surfaces, and Almost-K\"ahler Geometry}, 
J. Geom. Anal (2015) 25: 1744-1772. 

\bibitem{L}
C. LeBrun, 
\emph{Curvature in the Balance: The Weyl Functional and Scalar curvature of 4-Manifolds}, 
arXiv:2203.02528 [math. DG], to appear in Pure Appl. Math. Q. 


\bibitem{L}
M. Lejmi, 
\emph{Extremal almost-K\"ahler metrics}, 
Internat. J. Math., Vol. 21, no.12, 1639-1662 (2010).



\bibitem{aL}
A.-K. Liu,
\emph{Some new applications of general wall-crossing formula, Gompf's conjecture and its applications},
Math. Res. Lett., 3 (1996), 569-585.

\bibitem{MS}
D. McDuff and D. Salamon, 
\emph{Introduction to Symplectic Topology}, 
Oxford University Press, 1998. 






\bibitem{MM}
M. Micallef and J. D. Moore, 
\emph{Minimal two-spheres and the topology of manifolds with positive curvature
on totally isotropic two-planes}, 
Ann. of Math. 127 (1988), 199-227. 

\bibitem{OO}
H. Ohta and K. Ono, 
\emph{Notes on symplectic manifolds with $b_{+}=1$, $\Pi$},
Internat. J. Math., 7 (1996), 755-770.  


\bibitem{P}
P. Petersen, 
\emph{Riemannian geometry}, 
Graduate Texts in Mathematics, 171, Springer-Verlag, New York, 1998. 





\bibitem{T}
S. Tanno, 
\emph{4-dimensional conformally flat K\"ahler manifolds}, 
T$\hat{o}$hoku Math. Journ. 24(1972), 501-504. 




\bibitem{T}
C. H. Taubes, 
\emph{More constraints on symplectic forms from Seiberg-Witten invariants}, 
Math. Res. Lett., 2, 9-13 (1995), 


\bibitem{Y}
S. -T. Yau, 
\emph{Calabi's conjecture and some new results in algebraic geometry}, 
Proc. Nat. Acad. USA 74 (1977) 1798-1799.


\end{thebibliography}
\end{document}